\documentclass{birkjour}
\usepackage{amsmath,amssymb}

\usepackage{graphicx}

\usepackage[colorlinks=true,citecolor=black,linkcolor=black,urlcolor=blue]{hyperref}
\newtheorem{thm}{Theorem}[section]
\newtheorem{cor}[thm]{Corollary}
\newtheorem{lem}[thm]{Lemma}
\newtheorem{prop}[thm]{Proposition}

\theoremstyle{definition}

\newtheorem{rem}[thm]{Remark}

\numberwithin{equation}{section}

\usepackage{pgf,tikz}
\usetikzlibrary{arrows}

\begin{document}

\title[The automorphism group of the bipartite Kneser graph]
 {The automorphism group of the \\bipartite Kneser graph}

\author[S. Morteza Mirafzal]{S. Morteza Mirafzal}

\address{Department of Mathematics\\ Lorestan University\\ Khoramabad\\ Iran}

\email{smortezamirafzal@yahoo.com}
\email{mirafzal.m@lu.ac.ir}

\thanks{}


\date{}
\begin{abstract}Let $n$ and $k$ be integers with  $n>2k,  k\geq1$. We denote by $H(n, k)$ the $bipartite\  Kneser\
graph$, that   is,  a graph with the family of $k$-subsets and   ($n-k$)-subsets   of $[n]  = \{1, 2, ... , n\}$ as vertices, in which
 any two vertices are adjacent if and only if    one of them is a subset of the other.
In this paper, we   determine the automorphism group of $H(n, k)$. We show   that $Aut(H(n, k))\cong Sym([n]) \times \mathbb{Z}_2$ where $\mathbb{Z}_2$ is the cyclic group of order $2$. Then, as an application of the obtained result,  we  give a new proof for determining the automorphism group  of the Kneser graph $K(n,k)$.    In fact we show how to determine the automorphism group of the Kneser graph $K(n,k)$ given the automorphism group of the Johnson graph  $J(n,k)$. Note that the known proofs for determining the automorphism groups of Johnson graph $J(n,k)$ and Kneser graph $ K(n,k)$ are independent from each other.
\end{abstract}
\maketitle
\section{ Introduction}
\noindent
For a positive integer $n >  1 $,  let $[n]  = \{1, 2, ... , n\} $   and $V$ be the set of all $k$-subsets
and $(n-k)$-subsets of $[n]$.
 The $bipartite\  Kneser\  graph$ $H(n, k)$ has
$V$ as its vertex set, and two vertices $A,  B$ are adjacent if and only if $A \subset B$ or $B\subset A$. If $n = 2k$ it is obvious that
we do not have any edges, and in such a case, $H(n, k)$ is a null graph, and  hence we assume that  $n \geq 2k + 1$.
It follows from the definition of the graph $H(n, k)$ that it has
 2${n}\choose{k}$  vertices and the degree of each of its vertices  is
 ${n-k}\choose{k}$= ${n-k}\choose{n-2k}$, hence it is a regular graph. It is clear that $H(n, k)$ is a bipartite graph.
In fact,
if  $V_1=\{ v\in V(H(n ,k))   |  \  |v| =k \}$ and $V_2=\{ v\in V(H(n ,k))  | \  |v| =n-k \}$, then $\{ V_1, V_2\}$
is a partition of $V(H(n ,k))$ and every edge of $H(n, k)$ has a vertex in $V_1$ and a  vertex in $V_2$ and
$| V_1 |=| V_2 |$.  It is an easy task to show that   the graph $H(n, k)$  is a connected graph. The bipartite Kneser graph $H(2n+1, n)$ is known as the $middle\  cube$ $MQ_{2n+1} = {Q_{2n+1}}(n,n+1)$  [3]  or $regular\  hyperstar$ graph $HS(2(n+1),n+1)$ [11,13].

The regular hyperstar graph $ {Q_{2n+1}}(n,n+1) $  has been investigated from various aspects,  by various authors and some of the recent works about this class of graphs are [3,6,11,13,16,17].  The following figure shows the graph  $H(5,2)$ ($ Q_{5}(2,3) $) in plane. Note that in this figure the set $\{i,j,k\} $ ($\{ i,j \}$) is denoted by $ijk$ ($ij$). \

\definecolor{qqqqff}{rgb}{0.,0.,1.}
\begin{tikzpicture}[line cap=round,line join=round,>=triangle 45,x=.65cm,y=.80cm]
\clip(-4.3,-2.38) rectangle (11.32,6.3);
\draw (-0.9,3.74) node[anchor=north west] {13};
\draw (0.9,5.6) node[anchor=north west] {123};
\draw (4.88,5.3) node[anchor=north west] {124};
\draw (5.7,3.78) node[anchor=north west] {24};
\draw (5.7,1.9) node[anchor=north west] {245};
\draw (5.0,0.66) node[anchor=north west] {45};
\draw (3.0,0.04) node[anchor=north west] {453};
\draw (1.14,0.52) node[anchor=north west] {43};
\draw (-0.8,1.68) node[anchor=north west] {143};
\draw (3.28,5.9) node[anchor=north west] {12};
\draw (4.92,4.76)-- (3.06,5.52);
\draw (3.06,5.52)-- (1.14,4.88);
\draw (1.14,4.88)-- (0.12,3.52);
\draw (0.12,3.52)-- (0.22,1.64);
\draw (0.22,1.64)-- (1.24,0.64);
\draw (1.24,0.64)-- (3.44,0.2);
\draw (3.44,0.2)-- (4.88,0.64);
\draw (5.72,3.38)-- (5.6,3.34);
\draw (4.92,4.76)-- (5.6,3.34);
\draw (5.72,1.68)-- (5.66,1.82);
\draw (5.66,1.82)-- (5.6,3.34);
\draw (5.66,1.82)-- (4.88,0.64);
\draw (3.96,3.8)-- (3.24,4.48);
\draw (3.96,3.8)-- (4.02,1.8);
\draw (4.02,1.8)-- (3.44,1.28);
\draw (1.64,3.54)-- (1.66,1.82);
\draw (4.86,3.38)-- (4.8,1.88);
\draw (0.7,3.82) node[anchor=north west] {23};
\draw (0.6,2.12) node[anchor=north west] {234};
\draw (4.68,3.86) node[anchor=north west] {14};
\draw (4.3,1.8) node[anchor=north west] {145};
\draw (3.32,5.) node[anchor=north west] {125};
\draw (3.5,4.48) node[anchor=north west] {25};
\draw (3.52,2.44) node[anchor=north west] {235};
\draw (2.84,1.18) node[anchor=north west] {35};
\draw (1.96,1.4) node[anchor=north west] {135};
\draw (2.0,4.2) node[anchor=north west] {15};
\draw (4.86,3.38)-- (0.22,1.64);
\draw (5.6,3.34)-- (1.66,1.82);
\draw (2.62,3.54)-- (2.58,3.54);
\draw (2.58,3.54)-- (3.24,4.48);
\draw (2.7,1.66)-- (2.68,1.46);
\draw (2.68,1.46)-- (2.58,3.54);
\draw (3.44,1.28)-- (2.68,1.46);
\draw (1.64,3.54)-- (4.02,1.8);
\draw (1.66,1.82)-- (1.24,0.64);
\draw (2.68,1.46)-- (0.12,3.52);
\draw (1.64,3.54)-- (1.14,4.88);
\draw (2.58,3.54)-- (4.8,1.88);
\draw (3.24,4.48)-- (3.06,5.52);
\draw (4.92,4.76)-- (4.86,3.38);
\draw (3.96,3.8)-- (5.66,1.82);
\draw (4.8,1.88)-- (4.88,0.64);
\draw (3.44,1.28)-- (3.44,0.2);
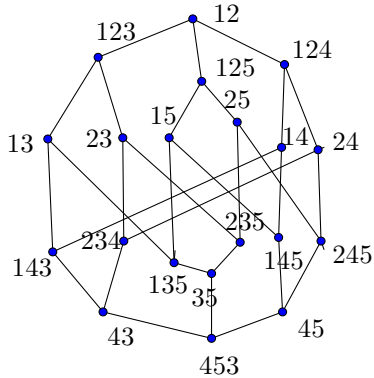
\draw (-1.9,-0.48) node[anchor=north west] {Fig 1. The bipartite Kneser graph  H(5,2)};
\begin{scriptsize}
\draw [fill=qqqqff] (0.12,3.52) circle (1.5pt);
\draw [fill=qqqqff] (0.22,1.64) circle (1.5pt);
\draw [fill=qqqqff] (1.24,0.64) circle (1.5pt);
\draw [fill=qqqqff] (3.44,0.2) circle (1.5pt);
\draw [fill=qqqqff] (4.88,0.64) circle (1.5pt);
\draw [fill=qqqqff] (1.14,4.88) circle (1.5pt);
\draw [fill=qqqqff] (3.06,5.52) circle (1.5pt);
\draw [fill=qqqqff] (4.92,4.76) circle (1.5pt);
\draw [fill=qqqqff] (5.6,3.34) circle (1.5pt);
\draw [fill=qqqqff] (5.66,1.82) circle (1.5pt);
\draw [fill=qqqqff] (1.64,3.54) circle (1.5pt);
\draw [fill=qqqqff] (1.66,1.82) circle (1.5pt);
\draw [fill=qqqqff] (4.86,3.38) circle (1.5pt);
\draw [fill=qqqqff] (4.8,1.88) circle (1.5pt);
\draw [fill=qqqqff] (3.24,4.48) circle (1.5pt);
\draw [fill=qqqqff] (3.44,1.28) circle (1.5pt);
\draw [fill=qqqqff] (3.96,3.8) circle (1.5pt);
\draw [fill=qqqqff] (4.02,1.8) circle (1.5pt);
\draw [fill=qqqqff] (2.58,3.54) circle (1.5pt);
\draw [fill=qqqqff] (2.68,1.46) circle (1.5pt);
\end{scriptsize}
\end{tikzpicture}\

 It was conjectured by Dejter, Erd\H{o}s, and Havel [6] among others, that the middle  cube ${Q_{2n+1}}(n,n+1)$
is Hamiltonian.
 Recently, M\"utze and Su [17]  showed that the bipartite Kneser graph $H(n, k)$ has a Hamilton cycle for all values of $k$. Among various interesting properties of the bipartite Kneser graph $H(n, k)$, we are  interested in its automorphism group and we want to know how this group acts on the  vertex set of $H(n, k)$. Mirafzal [13] determined the automorphism group of $ HS(2n,n)= H(2n-1, n-1)$  and  showed that $HS(2n,n)$ is a vertex-transitive non-Cayley graph. Also, he showed that $ HS(2n,n)$ is arc-transitive. \newline
 Some of the symmetry properties of the bipartite Kneser graph $H(n,k)$, are as follows.

\begin{prop}$[16, \  Lemma \ 3.1]$ The graph $H(n, k)$ is a vertex-transitive graph.

\end{prop}

\begin{prop}$[16, \ Theorem \  3.2] $ The graph $H(n, k)$ is a symmetric  (or arc-transitive) graph.
\end{prop}

\begin{cor}$[16, \  Corollary \ 3.3]$ The connectivity of the  bipartite Kneser graph $H(n, k)$ is
   maximum, namely,  ${n-k}\choose{k}$.

\end{cor}

\begin{prop}$[16, \  Proposition \ 3.5] $ The bipartite Kneser graph $H(n, 1)$ is a  Cayley graph.

\end{prop}

\begin{thm}$[16, \ Theorem \ 3.6] $  Let $H(n,1)$ be a bipartite Kneser graph.
 Then,   $Aut(H(n,1))   \cong Sym([n] ) \times \mathbb{Z}_2$, where $\mathbb{Z}_2$ is the cyclic group of order $2$.

\end{thm}

In $[16]$ the authors  proved the following theorem.

\begin{thm}$[16, \ Theorem \ 3.8]$  Let $n=2k-1$. Then, for the    bipartite Kneser graph $H(n,k-1)$, we have
  $Aut(H(n,k)) \cong Sym([n])  \times \mathbb{Z}_2$, where $\mathbb{Z}_2$ is the cyclic group of order $2$.

\end{thm}

In [16] the authors asked  the following question. \\\\
{ \bf Question } Is the above theorem true for all possible values of $n,k$ ( $2k <  n$)? \\

In the sequel, we want to answer the above question. We show that the above theorem is true for all possible values of $n,k$ ( $2k <  n$).  \

 In fact, to the best of our knowledge,  the present work is   the first answer on this problem.
We   determine the automorphism group  of the graph $H(n, k)$ and show that $Aut(H(n, k)) \cong Sym([n] ) \times \mathbb{Z}_2$,  where $\mathbb{Z}_2$ is the cyclic group of order $2$. In the final step of our work, we offer a new proof for determining the automorphism group of the Kneser graph $K(n,k)$ which we belief this proof is more elementary than other known proofs of this result. Note that the known proofs for determining the automorphism groups of Johnson graph $J(n,k)$ and Kneser graph $ K(n,k)$ are independent from each other.   we show   how we can have the automorphism group  of the Kneser graph $K(n,k)$ in the hand, if we have the automorphism group  of the Johnson  graph $J(n,k)$ in another hand. \\

There are various important families of graphs $\Gamma$,  in which we know that for a particular group $G$, we have
$G \leq Aut(\Gamma)$, but  showing  that in fact we have  $G = Aut(\Gamma)$, is a difficult task. For example note to the following cases. \newline

(1) \  The Boolean lattice $BL_n, n \geq 1$, is the graph whose vertex set is the set of all subsets of $[n]= \{ 1,2,...,n \}$, where two subsets $x$ and $y$ are adjacent if their symmetric difference has precisely one element. The hypercube  $Q_n$ is the graph whose vertex set is $ \{0,1  \}^n $, where two $n$-tuples  are adjacent if  they differ in precisely one coordinates. It is an easy task to show that $Q_n \cong BL_n $, and $ Q_n \cong Cay(\mathbb{Z}_{2}^n, S )$, where $\mathbb{Z}_{2}$ is
the cyclic group of order 2, and $S=\{ e_i \  | \  1\leq i \leq n \}, $ where  $e_i = (0, ..., 0, 1, 0, ..., 0)$,  with 1 at the $i$th position. It is an easy task to show that the set  $H= \{ f_\theta |\  \theta \in Sym([n]) \} $, $ f_\theta (\{x_1, ..., x_n \}) = \{ \theta (x_1), ..., \theta (x_n) \}$ is a subgroup of $Aut(BL_n)$, and hence $H$ is a subgroup of the group $Aut(Q_n)$.  We know that in every Cayley graph $\Gamma= Cay(G,S)$, the group $Aut(\Gamma)$ contains a subgroup isomorphic with the group $G$.  Therefore,  $\mathbb{Z}_{2}^n $ is a subgroup of $Aut(Q_n)$. Now,   showing that $Aut(Q_n) = <\mathbb{Z}_{2}^n, Sym([n])>( \cong \mathbb{Z}_{2}^n \rtimes Sym([n]))$, is  not an easy task [14]. \newline

(2) \  Let  $n,k \in \mathbb{ N}$ with $ k < \frac{n}{2}  $ and Let $[n]=\{1,...,n\}$.   The Kneser graph $K(n,k)$ is defined as the graph whose vertex set is $V=\{v\mid v\subseteq [n], |v|=k\}$ and two vertices $v$,$w$ are adjacent if and only if $|v\cap w|$=0.  The   Kneser graph $K(n,k)$  is  a vertex-transitive graph [5].  It is an easy task to show that the set  $H= \{ f_\theta \  |\  \theta \in Sym([n]) \} $,  $ f_\theta (\{x_1, ..., x_k \}) = \{ \theta (x_1), ..., \theta (x_k) \}$,  is a subgroup of  $ Aut ( K(n,k) )$ [5].   But,  showing  that
 $$ H= \{ f_\theta \  |\  \theta \in Sym([n]) \}= Aut ( K(n,k) )$$
 is rather a difficult  work [5, chapter 7]. \newline

(3) \   Let $n$ and $k$ be integers with  $n> k\geq1$ and let $[n]  = \{1, 2, ... , n\}$.  We now  consider the bipartite Kneser graph $\Gamma = H(n,k)$.    Let $A,B$ be $m$-subsets of $[n]$  and let  $ | A \cap B |=t$. Let  $\theta$ be  a permutation in $Sym([n])$. It is an easy task to show that
  $  | f_\theta(A) \cap  f_\theta(B) |=t$, where $ f_\theta (\{x_1, ..., x_m \}) = \{ \theta (x_1), ..., \theta (x_m) \}$.
 Moreover, if $ A\subset B$,  then $  f_\theta(A) \subset f_\theta(B) $. Therefore,  if $\theta \in Sym([n])$,  then
$$ f_\theta : V(H(n,k) )\longrightarrow V(H(n,k)),
 f_\theta (\{x_1, ..., x_k \}) = \{ \theta (x_1), ..., \theta (x_k) \} $$
 is an automorphism of $ H(n,k) $ and the mapping,   \newline
  $ \psi : Sym ([n]) \longrightarrow Aut (H(n,k) )$, defined by
 the rule $ \psi ( \theta ) = f_\theta $ is an injection. Therefore, the set  $H= \{ f_\theta \ |\  \theta \in Sym([n]) \} $,  is a subgroup of $ Aut ( H(n,k) ) $ which is isomorphic with $Sym([n])$.
 Also, the mapping $\alpha : V(\Gamma)\rightarrow V(\Gamma) $,  defined  by the rule, $\alpha(v) = v^c$,  where
$v^c$ is  the complement of the subset $v$ in  $[n]$, is an automorphism of the  graph $B(n,k)$. In fact,
if $ A\subset B$, then $ B^c\subset A^c$, and hence if \{A,B\} is an edge of the graph $ B(n,k) $, then $\{\alpha(A), \alpha(B)\}$ is an edge of the graph $ H(n,k) $.  Therefore we have,
$ < H, \alpha> \leq Aut(H(n,k)) $.  \newline
In this paper, we want to show that for the bipartite Kneser graph  $H(n,k)$,  in fact we have,  $ Aut(H(n,k))=<H, \alpha>$($ \cong  Sym([n])\times \mathbb{Z}_2)$.

\section{Preliminaries}
In this paper, a graph $\Gamma=(V,E)$ is
considered as a finite undirected simple graph where $V=V(\Gamma)$ is the vertex-set
and $E=E(\Gamma)$ is the edge-set. For all the terminology and notation
not defined here, we follow $[1,4,5]$.

The graphs $\Gamma_1 = (V_1,E_1)$ and $\Gamma_2 =
(V_2,E_2)$ are called $isomorphic$, if there is a bijection $\alpha
: V_1 \longrightarrow V_2 $   such that  $\{a,b\} \in E_1$ if and
only if $\{\alpha(a),\alpha(b)\} \in E_2$ for all $a,b \in V_1$.
In such a case the bijection $\alpha$ is called an isomorphism.
An automorphism of a graph $\Gamma $ is an isomorphism of $\Gamma
$ with itself. The set of automorphisms of $\Gamma$  with the
operation of composition of functions is a group, called the
$automorphism\  group$ of $\Gamma$ and denoted by $ Aut(\Gamma)$.

 The
group of all permutations of a set $V$ is denoted by $Sym(V)$  or
just $Sym(n)$ when $|V| =n $. A $permutation\  group$ $G$ on
$V$ is a subgroup of $Sym(V)$. In this case we say that $G$ act
on $V$. If $X$ is a graph with vertex-set $V$, then we can view
each automorphism as a permutation of $V$, and so $Aut(X)$ is a
permutation group. If $G$ acts on $V$, we say that $G$ is
$transitive$ (or $G$ $acts\  transitively$ on $V$), when there is just
one orbit. This means that given any two elements $u$ and $v$ of
$V$, there is an element $ \beta $ of  $G$ such that  $\beta (u)= v
$.

The graph $\Gamma$ is called $vertex$-$transitive$, if  $Aut(\Gamma)$
acts transitively on $V(\Gamma)$.  For $v\in V(\Gamma)$ and $G=Aut(\Gamma)$, the stabilizer subgroup
$G_v$ is the subgroup of $G$ consisting of all automorphisms that
fix $v$. We say that $\Gamma$ is $symmetric$ (or $arc$-$transitive$) if, for all vertices $u, v, x, y$ of $\Gamma$ such that $u$ and $v$ are adjacent, also, $x$ and $y$ are adjacent, there is an automorphism $\pi$ in $Aut(\Gamma)$ such that $\pi(u)=x$ and $\pi(v)=y$.

Let  $n,k \in \mathbb{ N}$ with $ k \leq  \frac{n}{2}$,  and let $[n]=\{1,...,n\}$. The $Johnson\  graph$ $J(n,k)$ is defined as the graph whose vertex set is $V=\{v\mid v\subseteq [n], |v|=k\}$ and two vertices $v$,$w$ are adjacent if and only if $|v\cap w|=k-1$.  The Johnson  graph $J(n,k)$ is a vertex-transitive graph [5].  It is an easy task to show that the set  $H= \{ f_\theta |\  \theta \in Sym([n]) \} $,  $f_\theta (\{x_1, ..., x_k \}) = \{ \theta (x_1), ..., \theta (x_k) \} $,    is a subgroup of $ Aut( J(n,k) ) $[5].   It has been shown that  $Aut(J(n,k)) \cong Sym([n])$, if  $ n\neq 2k, $  and $Aut(J(n,k)) \cong Sym([n]) \times \mathbb{Z}_2$, if $ n=2k$,   where $\mathbb{Z}_2$ is the cyclic group of order 2 [2,9,15]. \

 Although, in most situations  it is difficult  to determine the automorphism group
of a graph $\Gamma$ and how it acts on the vertex set of $\Gamma$,  there are various papers in the literature,   and some of the recent works
appear in the references [7,8,9,10,12,13,14,15,16,18,19].

\section{Main results}

\begin{lem} Let $n$ and $k$ be integers with  $\frac{n}{2}> k\geq1$, and
let $\Gamma= (V,E)= H(n,k)$ be a bipartite Kneser graph with partition $V=V_1 \cup V_2 $, $V_1 \cap V_2 =\varnothing$, where $ V_1= \{ v \  |  \  v\subset [n], |v|=k \}$ and  $ V_2= \{ w \ |\  w \subset [n], |w|=n-k \}$. If $f$ is an automorphism of $ \Gamma$ such that $f(v)=v$ for every $v\in V_1$, then $f$ is the identity automorphism of $ \Gamma$.

\end{lem}

\begin{proof}

First, note that since $f$ is a permutation of the vertex set $V$
and $f(V_1)=V_1$, then $f(V_2)= V_2$.
Let $ w\in V_2$ be an arbitrary vertex in $V_2$. Since $f$ is an automorphism of the graph $\Gamma$, then for the set  $N(w)= \{ v |  v\in V_1, v\leftrightarrow w \}$, we have $f(N(w))= \{ f(v) |  v\in V_1, v\leftrightarrow w \}=N(f(w))$. On the other hand, since for every $v\in V_1$, $f(v)=v$, then
$f(N(w))=N(w) $, and therefore $N(f(w))=N(w) $. In other words,  $w$ and $f(w)$ are $(n-k)$-subsets of $[n]$ such that
their family of $k$-subsets are the same. Now, it is an easy task to show that $f(w)=w$. Therefore, for every vertex $x$ in $\Gamma$  we have $f(x)=x$ and thus $f$ is the identity automorphism of $\Gamma$.

\end{proof}

\begin{rem} If in the assumptions of the  above lemma,  we replace  with  $f(v)=v$ for every $v\in V_2$, then we can show, by a similar discussion,   that $f$ is the identity automorphism of $ \Gamma$.
\end{rem}

\begin{lem} Let $\Gamma =(V, E)$ be a connected  bipartite graph with partition $V=V_1 \cup V_2$, $V_1 \cap V_2 = \varnothing $. Let $ f$ be an automorphism of $\Gamma $ such that for a fixed vertex $v \in V_1 $, we have $ f(v) \in V_1$. Then, $f(V_1) = V_1$ and $f(V_2) =V_2$. Or, for a fixed vertex $v \in V_1 $,   we have $ f(v) \in V_2$. Then, $f(V_1) = V_2$ and $f(V_2) =V_1$.
\end{lem}

\begin{proof}
In the first step, we show that if $ w \in V_1 $ then $f(w) \in V_1$. We know that if $ w\in V_1$, then $ d_\Gamma (v, w) = d(v, w)$, the distance between $ v$ and $ w$ in the graph $ \Gamma$,  is an even integer.
Assume $ d(v, w) =2l$, $ 0\leq 2l \leq D$, where $ D$ is the diameter of $\Gamma $. We  prove by induction on $ l$,  that $ f(w) \in V_1$. If $ l=0$,  then $ d(v, w) =0$, thus $v=w$, and hence $f(w)=f(v) \in V_1$.
 Suppose that if  $ w_1 \in V_1$ and  $ d(v, w_1)= 2(k-1)$,  then $ f(w_1) \in V_1$.
Assume $ w \in V_1$ and $d(v, w)=2k $.  Then, there is a vertex $ u \in \Gamma $ such that
$d(v, u)=2k-2=2(k-1)$ and $ d(u, w)=2$.
We know  (by the induction assumption) that  $ f(u) \in V_1$  and since $ d(f(u),f(w))=2$, therefore  $f(w) \in V_1 $. Now, it follows that $ f(V_1)=V_1$ and  consequently   $ f(V_2)=V_2$.
\end{proof}

\begin{cor}
Let $\Gamma =H(n,k) = (V,E) $ be a bipartite Kneser graph with partition $V=V_1 \cup V_2$, $V_1 \cap V_2 = \varnothing  $. If $f$ is an automorphism of the graph $\Gamma$,  then $ f(V_1)=V_1$ and $f(V_2) =V_2$,  or $ f(V_1) = V_2 $ and $f(V_2) = V_1$.

\end{cor}

In the sequel, we need the following result for proving our  main  theorem.

\begin{lem}
Let $l,m,u $ are positive integers with $ l > u$ and $ m> u$. If $l>m$ then ${l} \choose {u}$ $>$ ${m} \choose {u}$.

\end{lem}

\begin{proof}
The proof is straightforward.
\end{proof}

\begin{thm}

Let $n$ and $k$ be integers with  $\frac{n}{2}> k\geq1$, and
let $\Gamma= (V,E)= H(n,k)$ be a bipartite Kneser graph with partition $V=V_1 \cup V_2 $, $V_1 \cap V_2 = \varnothing$, where $ V_1= \{ v \  | \  v\subset [n], |v|=k \}$ and  $ V_2= \{ w \  |  \ w \subset [n], |w|=n-k \}$. Then, $Aut(\Gamma) \cong Sym([n]) \times \mathbb{Z}_2$,  where $\mathbb{Z}_2$ is the cyclic group of order $2$.

\end{thm}

\begin{proof}Let $\alpha : V(\Gamma)\rightarrow V(\Gamma) $,  defined  by the rule, $\alpha(v) = v^c$,  where
$v^c$ is  the complement of the subset $v$ in  $[n]$.
 Also, let  $ H =\{ f_\theta |\  \theta \in Sym ([n]) \} $,  $f_\theta (\{x_1, ..., x_k \}) = \{ \theta (x_1), ..., \theta (x_k)$.  We have seen already that $H( \cong Sym([n]))$
 and $<\alpha>( \cong \mathbb{Z}_2)$ are subgroups of the group $G= Aut(\Gamma)$.
 We can see that $ \alpha \not\in H $,   and for every $\theta \in Sym([n])$, we have,   $f_{\theta}\alpha= \alpha f_{\theta}$ [15]. Therefore,   \
$$  Sym([n]) \times \mathbb{Z}_2  \cong H \times <\alpha>\cong <H, \alpha> $$
$$ =\{  f_{\gamma} \alpha^i\ |   \   \gamma \in Sym([n]), 0\leq i \leq 1 \}=S$$ \
 is a subgroup of $G$. We now  want to show that $G=S$.  Let $f \in Aut(\Gamma)=G$. We show that $ f \in S$. There are two cases\newline
 (i)  There is a vertex $v \in V_1$   such that $f(v) \in V_1$, and hence by Lemma 3.3. we have    $f(V_1)=V_1$.\newline
 (ii)  There is a vertex $v \in V_1$   such that $f(v) \in V_2$, and hence by Lemma 3.3. we have    $f(V_1)=V_2$.  \

\

\ \ (i)  Let $f(V_1)=V_1$. Then, for every vertex $v \in V_1$ we have $f(v) \in V_1$, and therefore the mapping $ g=f_{|V_1}: V_1 \rightarrow  V_1$, is a permutation of $V_1$ where $ f_{|V_1}  $ is the restriction of $f$ to $V_1$. Let $ \Gamma_2= J(n,k)$ be the Johnson graph with the vertex set $V_1$. Then, the vertices $ v,w \in V_1$ are adjacent in $\Gamma_2$  if and only if $| v \cap w| =k-1$.\

We assert that the permutation  $ g= f_{|V_1}  $ is an automorphism of the graph $\Gamma_2$. \newline
For proving our assertion, it is sufficient to show that if $v,w \in V_1$ are such that $| v \cap w| =k-1$ then we have  $ |g(v) \cap g(w) |= k-1$. Note that since $v,w$ are $k$-subsets of $[n]$, then if $u$ is a common neighbour
of $v,w$ in the bipartite Kneser graph $ \Gamma=H(n,k)$, then the set $u$ contains the sets $v$ and $w$. In particular $u$ contains the $(k+1)$-subset $ v\cup w$. We now can see that the number of vertices $u$, such that $u$ is adjacent in $\Gamma$ to both of the  vertices $v$ and $w$, is ${n-k-1}\choose  {n-2k-1}$.   Note that if $t$  is a positive integer such that $ k+1+t=n-k $,  then $t= n-2k-1$. Now,  if we
 adjoin to the $(k+1)$-subset $v \cup w$ of $[n]$, $n-2k-1$ elements  of the complement of $ v \cup w $ in $[n]$, then we obtain a subset $u$ of $[n]$ such that $v \cup w \subseteq u$ and $u$ is a $(n-k)$-subset of $[n]$.  Now, since $v$ and $w$ have $ {n-k-1} \choose{n-2k-1} $ common neighbours in the graph
 $\Gamma$, then the vertices $g(v)$ and $g(w)$ must have
${n-k-1} \choose{n-2k-1}$=${n-k-1} \choose{k}$ neighbours in $\Gamma$, and
 therefore $| g(v) \cap g(w) | $= $k-1$.
In fact, if $|g(v) \cap g(w)|= k-h < k-1 $, then $ h>1$ , and hence $ |g(v) \cup g(w)| = k+h $. Thus, if  $t$  is a positive integer such that $ k+h+t=n-k $,  then $t= n-2k-h$. Hence, for constructing  a $(n-k)$-subset $u \supseteq g(v) \cup g(w) $ we must adjoin $t=n-2k-h$ elements  of the complement of $ g(v) \cup g(w) $   in $[n]$,  to the set $ g(v) \cup g(w) $.  Therefore the number of  common neighbours of vertices $g(v)$ and  $ g(w) $  in the graph $\Gamma$ is
 ${n-k-h} \choose{n-2k-h}$=${n-k-h} \choose{k}$. Note that by Lemma 3.5. it follows that  ${n-k-h} \choose{k}$  $\neq$  ${n-k-1} \choose{k}$. \

Our argument shows that  the permutation $g=f_{|V_1}$ is an automorphism of the Johnson graph $ \Gamma_2=J(n,k) $ and therefore by [2 chapter 9, 15] there is a permutation $ \theta \in Sym([n]) $
such that $g= f_{\theta}$. \

On the other hand, we know that $ f_{\theta} $  by its natural action  on the vertex set of the bipartite Kneser graph $ \Gamma= H(n,k) $  is an automorphism of $ \Gamma$.   Therefore, $l=f_{\theta }^{-1}f  $ is an automorphism of the bipartite  Kneser graph $\Gamma$
such that $l$ is the identity automorphism on the subset $V_1$. We now can   conclude, by Lemma 3.1.  that $l= f_{\theta }^{-1}f$,  is the identity automorphism of $ \Gamma $, and therefore $f=f_{\theta }$. \

In other words, we have proved that if $f$ is an automorphism of $\Gamma = H(n,k)$ such that $f(V_1)=V_1$, then $f=f_{\theta}$, for some $ \theta \in Sym([n] )$, and hence $f \in S$.  \

\ (ii) We now assume that $ f(V_1) \neq V_1 $. Then,  $f(V_1) = V_2$. Since the mapping $ \alpha $  is an automorphism of the graph $ \Gamma $, then $f \alpha$ is an automorphism of $\Gamma $ such that $f \alpha(V_1) = f(\alpha(V_1))= f(V_2)=V_1$. Therefore, by what is proved in (i), we have $ f \alpha=f_{\theta} $, for some $ \theta \in Sym([n]) $. Now since $ \alpha    $  is of order $2$, then $f= f_{\theta} \alpha \in S=\{  f_{\gamma} \alpha^i \ | \  \gamma \in Sym([n]), 0\leq i \leq 1 \}$.

\end{proof}

Let $n,k$ be integers and  $n >4, \      k< \frac {n}{2}, \  [n]=\{1,2,...,n   \}$.  Let  $K(n,k)= \Gamma$
be a Kneser graph. It is a well known fact that,    $Aut (\Gamma)\cong Sym([n])$ [5, chap 7]. In fact,
 the proof in [5, chap 7] shows that the automorphism group of the Kneser graph $K(n,k)$ is the group
 $ H =\{ f_\theta \  |\  \theta \in Sym ([n]) \}(\cong Sym([n])) $.
 The  proof of  this result,  that appears in [5, chap 7],   uses the following fact which is one of  the  fundamental results in extermal set theory. \

\

{\bf Fact} (Erd\H{o}s-Ko-Rado)   If $ n> 2k $, then $ \alpha (K(n,k)) $ = $ {n-1} \choose {k-1}$, where $ \alpha (K(n,k)) $ is the independence number of the Kneser graph $K(n,k)$.

\

In the sequel, we provide a new proof for determining the automorphism groups of Kneser graphs.  The main  tool which we use in our method  is Theorem 3.6. Note that,  the main tool which we use  for proving Theorem 3.6. is  the automorphism group of Johnson graph $J(n,k)$, which have been already obtained [2 chapter  9, 15]  by using elementary and relevant facts of graph theory and group theory.

\begin{thm}
Assume $n,k$ are  integers and  $n >4, \      k< \frac {n}{2}, \  [n]=\{1,2,...,n   \}$. If $K(n,k)= \Gamma$ is a Kneser graph, then we have $Aut (\Gamma)\cong Sym([n])$.

\end{thm}
\begin{proof}
Let $g$ be an automorphism of the graph $\Gamma$. We now consider the bipartite Kneser graph $ \Gamma_1=H(n,k)=( V,E) $,  with partition $V=V_1 \cup V_2 $, $V_1 \cap V_2 =\varnothing$, where $ V_1= \{ v \  | \  v\subset [n], |v|=k \}$ and  $ V_2= \{ w \  |  \ w \subset [n], |w|=n-k \}$. We define the mapping $ f: V \rightarrow  V  $ by the following rule; \

 $$ f(v) =  \begin{cases}
g(v) \ \ \ \ \ \ \ \ \ \  v \in V_1  \\  (\alpha g \alpha) (v) \ \ \ \ v \in V_2\\
 \end{cases} $$ \

It is an easy task to show that $f$ is a permutation of the vertex set $V$ such that $f(V_1) = V_1 = g(V_1)$. We show that $f$ is an automorphism of the bipartite Kneser graph $\Gamma_1$. Let $\{v,w\}$ be an edge of the graph $\Gamma_1$ with $v \in V_1$. Then $v \subset w $, and hence $ v \cap w^c=v \cap  \alpha(w) = \varnothing$. In other words  $ \{v , \alpha(w) \}$  is an edge of the Kneser  graph $\Gamma$.  Now, since the mapping $g$ is an automorphism of the Kneser graph $\Gamma$,  then $ \{g(v),  \ g(\alpha (w)) \}$ is an edge of the  Kneser graph $\Gamma$, and therefore we have $ g(v)\cap g(\alpha (w)) =  \varnothing$.  This implies that $ g(v) \subset {(g(\alpha (w))) }^c=\alpha( g(\alpha (w)) )$. In other words $ \{g(v),  \alpha(g(\alpha (w)) ) \} =  \{f(v), f(w)\}$ is an edge of the bipartite Kneser graph $\Gamma_1$.
 Therefore $f$ is an automorphism of the bipartite Kneser graph $H(n,k)$. Now, since $f({V_1}) =V_1$, then by Theorem 3.6. there is a permutation $\theta$ in $Sym([n])$ such that $f= f_{ \theta}$. Then,  for every $v\in V_1$ we have $ g(v)= f(v)=f_{\theta}(v)$, and therefore $g=f_{\theta} $.
We  now can conclude that $ Aut( K(n,k))$ is a subgroup of the group $ H =\{ f_\gamma \  |\  \gamma \in Sym ([n]) \} $. On the other hand, we can see that  $H$ is a subgroup of $Aut(K(n,k))$, and therefore we have $Aut(K(n,k))=H =\{ f_\gamma  \  | \  \gamma \in Sym ([n]) \} \cong Sym ([n])$.

\end{proof} \

\section{ Conclusion}
In this paper, we     studied one of the algebraic properties of the bipartite Kneser graph $H(n,k)$. We determined
the automorphism group of this graph for all  $n,k,$ where $2k < n$
(Theorem 3.6). Then, by Theorem 3.6.  we offered a new proof for determining the automorphism group of the Kneser graph $K(n,k)$(Theorem 3.7).

\section{ Acknowledgements}
The author is thankful to the anonymous reviewers  for their valuable comments and suggestions.

\end{document}